\documentclass{amsart}

\usepackage{amssymb, amsmath}
\usepackage{mathrsfs}
\usepackage{amscd}
\usepackage{verbatim}
\usepackage{enumerate}

\usepackage{hyperref}

\theoremstyle{plain}
\newtheorem{theorem}{Theorem}[section]
\theoremstyle{remark}
\newtheorem{remark}[theorem]{Remark}

\newtheorem{example}[theorem]{Example}

\theoremstyle{plain}

\newtheorem{lemma}[theorem]{Lemma}
\newtheorem{proposition}[theorem]{Proposition}
\newtheorem{definition}[theorem]{Definition}

\numberwithin{equation}{section}

\def\N{{\mathbb N}}

\def\R{{\mathbb R}}
\def\C{{\mathbb C}}

\newcommand{\E}{{\mathbb E}}
\renewcommand{\P}{{\mathbb P}}

\newcommand{\F}{{\mathcal F}}

\renewcommand{\Re}{\operatorname{Re}}

\newcommand{\calL}{{\mathcal L}}
\newcommand{\n}{\Vert}

\newcommand{\lb}{\langle}
\newcommand{\rb}{\rangle}

\begin{document}

\author{Roland Schnaubelt}
\address{R. Schnaubelt, Department of Mathematics \\
Karlsruhe Institute of Technology\\
D-76128  Karls\-ruhe\\Germany}
\email{schnaubelt@kit.edu}

\dedicatory{Dedicated to Jan Pr\"uss, in admiration of his unfaltering spirit \\ and his outstanding mathematical work.}

\author{Mark Veraar}
\address{M. Veraar, Delft Institute of Applied Mathematics\\
Delft University of Technology \\ P.O. Box 5031\\ 2600 GA Delft\\The
Netherlands} \email{M.C.Veraar@tudelft.nl}

\thanks{The second author
is supported by the VIDI subsidy 639.032.427 of the Netherlands Organisation for Scientific Research (NWO)}


\title
[Stochastic Volterra equations]
{Regularity of stochastic Volterra equations \\ by functional calculus methods}

\begin{abstract}
We establish pathwise continuity properties of solutions to a stochastic Volterra equation
with an additive noise term given by a local martingale. The deterministic part is governed
by an operator with an $H^\infty$-calculus and a scalar kernel. The proof relies on
the dilation theorem for positive definite operator families on a Hilbert space.
\end{abstract}



\maketitle

\section{Introduction}

In this paper we investigate pathwise continuity properties of the solutions to the stochastic Volterra equation
\begin{equation}\label{eq:Volterra}
u(t) = u_0+\int_0^t a(t-s) Au(s)\, ds + L(t), \ \ \ t\geq 0.
\end{equation}
Here $A$ is a closed and densely defined operator on a Hilbert space $(X, \lb\cdot, \cdot\rb)$, the kernel
$a$ belongs to $L^1_{\rm loc}(\R_+)$, and $L$ is an $X$-valued local $L^2$-martingale with c\`adl\`ag (or continuous) paths.
Stochastic Volterra equations are widely studied and we refer the reader to
\cite{BGK,CKK,ClDP96,ClDP97,ClDP00,CDPP,DeschLonden,DeLo09,Kar07} and the references given there.
In this paper we show that for a large class of kernels $a$ and operators $A$, there exists a version of the solution $u$
of \eqref{eq:Volterra} for which the paths are c\`adl\`ag (or continuous) using dilation theory, $H^\infty$-calculus
and the theory of deterministic Volterra equations.

Volterra equations arise in physical models whose constitutive laws depend on the history of the material.
Such behavior occurs in viscoelastic fluids or solids, in heat conduction with memory, or in electromagnetism.
In accordance with the theory in Pr\"uss' monograph  \cite{Pruss93} we look at an integrated formulation
of such problems, which  fits well to stochastic evolution equations.
The stochastic term $L(t)$ can be understood as a time integral of a given external random force
or a heat supply. We refer to Chapter~5 of  \cite{Pruss93} for a discussion of the underlying deterministic
models.

For more regular paths one could apply in \eqref{eq:Volterra} the theory developed in
\cite{Pruss93} pathwise (under appropriate conditions on $a$).
For instance, if $L$ has H\"older continuous paths and if the deterministic part
of \eqref{eq:Volterra} is of parabolic type  in the sense of \cite {Pruss93}, then $u$ also has H\"older
continuous paths by Theorem~3.3 of  \cite{Pruss93}.  However, for general local $L^2$-martingales H\"older continuity
is quite restrictive and even impossible if jumps occur. On the other hand, Chapter~8 of \cite{Pruss93} provides
a theory of maximal $L^p$-regularity for the deterministic part in the parabolic case, but it would only yield
$L^p$-properties of the paths, see Theorem~8.7 of \cite{Pruss93}.

In \cite{PZ14}, Peszat and Zabczyk found new conditions on $a$ under which the solution $u$ to \eqref{eq:Volterra}
has c\`adl\`ag (or continuous) trajectories. Their method is based on dilation results, which were previously
used in the case that $a=1$ and $A$ is the generator of a semigroup which satisfies $\|T(t)\|\leq e^{w t}$ for
all $t\geq 0$ and a fixed $w\in \R$, see \cite{HausSei08} and references therein. In Theorem~1 of \cite{PZ14}, it is assumed that $A$
is a self-adjoint operator so that the spectral theorem provides a functional calculus for $A$. The functional
calculus allows to reduce the problem to scalar Volterra equations.

However, many operators arising in applications fail to be selfadjoint.
For instance, an elliptic operator $A$ with $D(A) = H^2(\R^d)$ in non-divergence form
\[Au = \sum_{i,j=1}^d a_{ij} D_i D_ju + b_i D_iu + c u\]
with space-dependent coefficients $a_{ij}$, $b_i$ and $c$ is not selfadjoint in general.
In the system case, self-adjointness is even more problematic. Indeed, let $D(A) = H^2(\R^d;\R^{N})$ and
\[Au = \left(
        \begin{array}{ccc}
          A_{11} & \cdots & A_{1N} \\
          \vdots & \vdots & \vdots \\
          A_{N1} & \cdots & A_{NN}  \\
        \end{array}
      \right) \left(
                \begin{array}{c}
                  u_1 \\
                  \vdots \\
                  u_N  \\
                \end{array}
              \right)
\]
and assume that each $A_{mm}$ is itself an elliptic second order differential operator. Even
if the elliptic operators $A_{mn}$ have $x$-independent coefficients, the operator $A$ will only
lead to a self-adjoint operator if $A_{mn} = A_{nm}$ which is rather restrictive.

Under suitable ellipticity conditions and regularity assumptions on the coefficients the above
two operators possess a bounded $H^\infty$-calculus. During the last 25 years there has been a lot of progress in
the investigation of this functional calculus.  Originally it was developed by McIntosh and collaborators to solve
the Kato square root problem (see \cite{KATOproblem, Mc86}).  By now the $H^\infty$-calculus is well-established
and has become one of the central tools in operator-theoretic approaches to PDE. Any reasonable elliptic or
accretive differential operator on a Hilbert space admits a bounded $H^\infty$-calculus.

In this paper, $H^\infty$-calculus techniques allow us to show that
the solution of \eqref{eq:Volterra} has the same pathwise
continuity properties as the local $L^2$-martingale $L$, thereby covering the above indicated examples.
Besides the  $H^\infty$-calculus of the main operator $A$ we mainly assume sector conditions of the Laplace transform
of the kernel $a$, see Theorem~\ref{thm:main}.
The  $H^\infty$-calculus is first used to construct the solution operator (called \emph{resolvent}) of the  deterministic
Volterra equation with  $L=0$ by means of the solutions to a corresponding scalar problem with $A$ replaced by complex
numbers in a suitable sector. Thanks to Laplace transform techniques from \cite{Pruss93}, we can  derive the uniform
estimates on these solutions which are required to apply the $H^\infty$-calculus. Second, one employs the calculus
to check that the resolvent is positive definite in order to use the dilation argument from \cite{PZ14}.
In this step we also invoke  a different dilation  result taken from \cite{KunWeis04}. An important technical feature
are rescaling arguments which are needed since in applications usually only a shifted operator is known
to possess an $H^\infty$-calculus. We further discuss auxiliary facts, as well as examples for operators $A$ and
kernels $a$ in the second and the last section.

The $H^\infty$-calculus has already played an important role in several other works on stochastic partial
differential equations. In \cite{SeidlerLondon,VW11} it is used to derive maximal estimates for stochastic
convolutions by a dilation argument. Solutions with paths in $D((-A^{1/2}))$ almost surely are obtained via
square function estimates   in \cite{BD13,DeschLonden,DNW,NVW10,NVWgamma,Ver10,VeZi}. More indirectly,
characterizations of the $H^\infty$-calculus have already been employed in Theorem~6.14 of
\cite{DaPratoZabnewversion} and in \cite{Brz97}, in the form that $D_A(\theta,2) = D((-A)^{\theta})$ for some
$\theta\in (0,1)$ and that  $(-A)^{is}$ is bounded for all $s\in\R$, respectively.

\section{Preliminaries}

\subsection{Volterra equations}
We first recall a basic definition in the theory of Volterra equations of scalar type,
see Pr\"uss' monograph \cite{Pruss93} for details. For $\phi\in (0,\pi]$ we define the sector $\Sigma_\phi$ by
\[\Sigma_\phi = \{z\in \C\setminus\{0\}: |\arg(z)|<\phi\}.\]
Let $(A,D(A))$ be a closed, densely defined and injective operator on a Hilbert space $X$
and let $\sigma(A)$ denote its spectrum.
Such an operator $A$ is called \emph{sectorial of angle $\phi$} if
$\sigma(A)\subseteq \C\setminus\Sigma_{\pi-\phi}$ and there is a constant $C$ such that
$$ \| (\lambda -A)^{-1}\|_{\calL(X)}\leq  \frac{C}{|\lambda|}\qquad \text{whenever} \ \ \arg(\lambda) < \pi-\phi .$$
We further write $\phi_A$ for the infimum of all $\phi$ such that $A$ is sectorial of angle $\phi$.

Let $a\in L^1_{\rm loc}(\R_+)$ and $(A,D(A))$ be a closed operator. We study the Volterra equation
\begin{equation}\label{eq:volterraDet}
u(t) = f(t) + \int_0^t a(t-s) Au(s)\, ds, \qquad t\ge0,
\end{equation}
for a given measurable map $f:\R_+\to X$. Usually we extend $a$, $u$ and $f$ by zero on $(-\infty,0)$.
We then write \eqref{eq:volterraDet} as $u = f + a A*u$, where $*$ stands for the convolution. From
\cite{Pruss93} we recall the basic concept describing the solution operators of  \eqref{eq:volterraDet}.

\begin{definition}
A family $(S(t))_{t\geq 0}$ of bounded linear operators on $X$ is a called a {\em resolvent} for
\eqref{eq:volterraDet} if is satisfies the following conditions.
\begin{enumerate}[\rm (i)]
\item $S$ is strongly continuous on $[0,\infty)$ and $S(0)=I$.
\item We have $S(t)D(A)\subseteq D(A)$ and  $AS(t)x = S(t)Ax$ for all $t\geq 0$ and $x\in D(A)$.
\item The {\em resolvent equation}
\begin{equation}
S(t) x = x + \int_0^t a(t-s) AS(s)x\, ds
\end{equation}
is valid for all $x\in D(A)$ and $\geq 0$.
\end{enumerate}
\end{definition}
By Corollary~1.1 of \cite{Pruss93}, the problem \eqref{eq:volterraDet} possesses at most one resolvent.

\subsection{Functional calculus\label{subsec:funccalc}}
In this section we briefly discuss  the $H^\infty$-calculus which was developed by McIntosh \cite{Mc86} and many
others. We  also present important classes of examples below. For details we refer to \cite{Haase:2},
\cite{KunWeis04} and the references therein.

Let $H^\infty(\Sigma_\phi)$ denote the space of all bounded analytic functions
$f:\Sigma_\phi\to \C$, and $H_0^\infty(\Sigma_{\phi})$ be the subspace of all
$f\in H^\infty(\Sigma_{\phi})$ for which there exist $\varepsilon>0$ and $c\ge 0$ such that
$$ |f(z)| \le \frac{c\,|z|^\varepsilon}{(1+|z|)^{2\varepsilon}}, \quad z\in \Sigma_{\phi}.$$
If $A$ is sectorial, then for all $\phi_A<\phi'<\phi<\pi$ and
$f\in H_0^\infty(\Sigma_{\phi})$ we can define an operator $f(-A)$ in $\calL(X)$ by setting
$$ f(-A) = \frac1{2\pi i}\int_{\partial \Sigma_{\phi'}} f(z) (z+A)^{-1}\,dz.$$
We say that $-A$ has a {\em  bounded $H^\infty$-calculus} if there is a constant $C_A\ge 0$ and an angle
$\phi>\phi_A$ such that for all $f\in H_0^\infty(\Sigma_{\phi})$ we have
$$ \n f(-A)\n \le C_A\,\n f\n_{H^\infty(\Sigma_\phi)}.$$
We work here with  $-A$ instead of $A$ to be in accordance with \cite{Pruss93}, where dissipative
operators instead of accretive operators are used.

Clearly, every self-adjoint operator of negative type has a bounded $H^\infty$-calculus.
We next give several examples of more general situations.

\begin{example}[Dissipative operators]
Let $(A, D(A))$ be linear and injective such that  $I-A$ is invertible.
Assume that $A$ is $\phi$-dissipative for some $\phi\in [0,\pi/2]$; i.e.,
\[|\arg(\lb A x, x\rb)|\geq \pi-\phi, \qquad x\in D(A).\]
It is well known that then $A$ is sectorial with $\phi_A\leq \phi$. Moreover,
$-A$ has a bounded $H^\infty$-calculus by e.g.\ Theorem~11.5 in \cite{KunWeis04}.

A special case of this situation are
normal operators with spectrum in $\C\setminus\Sigma_{\pi-\phi}$.
\end{example}

\begin{example}[BIP]
Let $(A, D(A))$ be a sectorial operator which has
 bounded imaginary powers (BIP); i.e., $(-A)^{is}\in \calL(X)$ for all $s\in \R$.
Then $-A$ has a bounded $H^\infty$-calculus,  see e.g.\ Theorem~11.9  in \cite{KunWeis04}. We recall that
$-A$ has bounded imaginary powers if and only if $D((-A)^{\theta}) = D_A(\theta, 2)$ for some $\theta\in (0,1)$,
where the  latter is the real interpolation space between $X$ and $D(A)$, see
Sections~6.6.3 and 6.6.4 in \cite{Haase:2}.
\end{example}

\begin{example}[Elliptic operators]
Let $A = \sum_{m,n=1}^d a_{mn} \partial_{m}\partial_n + \sum_{n=1}^d b_n \partial_n+ c$ with coefficients
 $b_n,c\in L^\infty(\R^d)$ and $a_{mn}\in C^{\varepsilon}_b(\R^d)$ for some $\varepsilon>0$.
For some  $\phi\in (0,\pi]$ we assume that
\[\sum_{m,n=1}^N a_{mn}(x)\xi_n \xi_m \in \{\lambda\neq 0: |\arg(\lambda)|\geq \phi\},
            \qquad \xi\in \R^d\setminus\{0\}, x\in \R^d.\]
Then for all $\phi'>\phi$ there exists a $w$ such that $A-w$ is sectorial of angle $\phi'$ and $-(A-w)$ has a
bounded $H^\infty$-calculus, see e.g.\ Theorem~13.13 in \cite{KunWeis04}.
\end{example}

The above list is far from exhaustive. For other results on operators with a
bounded $H^\infty$-calculus we refer the reader to \cite{CDMY96, DenketalSeeley, DuSim}
and to \cite{KATOproblem} for connections to the famous Kato square root problem.

\section{The main result}

Let $X$ be a separable Hilbert space and $(\Omega,\mathcal{A},\P)$ be a complete probability space
with a filtration $(\mathcal{F}_{t})_{t\geq 0}$ satisfying the usual conditions (see \cite{Kal}).
We assume that $L$ is an $X$-valued local $L^2$-martingale with c\`adl\`ag paths almost surely, that the kernel $a$
belongs to  $L^1_{\rm loc}(\R_+)$, and that $A$ is a closed operator on $X$ with dense domain $D(A)$.
 We study the stochastic Volterra equation
\begin{equation}\label{eq:Volterra2}
u(t) = u_0 + \int_0^t a(t-s) Au(s)\, ds + L(t), \ \ \ t\geq 0,
\end{equation}
for an  $\mathcal{F}_0$-measurable initial function $u_0:\Omega\to X$. We may assume that
$L(0)=0$ as we could replace $u_0$ by $u_0+L(0)$.

In Theorem~\ref{thm:main}  we present the main result on the existence of solutions with regular paths.
Several classes of admissible kernels $a$ will be discussed in Section~\ref{subsec:aex}. Finally, in
Section~\ref{subsec:illustrate} we illustrate how the results of Sections~\ref{subsec:funccalc} and
\ref{subsec:aex} can be combined with Theorem~\ref{thm:main} to obtain path properties of solutions.

Before we move to the main result we first give the definition of a weak solution to \eqref{eq:Volterra2}
and we show a simple but useful lemma about shifting the operator $A$.
\begin{definition}
A measurable process $u:\R_+\times\Omega\to X$  is called a \emph{weak solution} to \eqref{eq:Volterra2} if almost
all paths of $u$  belong to $L^1_{\rm loc}(\R_+;X)$ and if for all $x^*\in D(A^*)$ and $t\in [0,\infty)$
we have  almost surely
\[\lb u(t), x^*\rb = \lb u_0, x^*\rb  + \int_0^t a(t-s) \lb u(s), A^*x^*\rb \, ds + \lb L(t), x^*\rb.\]
\end{definition}
Assume that the resolvent for \eqref{eq:volterraDet} exists. Proposition~2 of \cite{PZ14} then says that there is a
unique weak solution $u$ of \eqref{eq:Volterra2} given by
\begin{equation}\label{eq:mild}
 u(t) = S(t) u_0  + \int_0^t S(t-s) \, dL(s).
\end{equation}
Here the stochastic integral exists since $S$ is strongly continuous
and $L$ is a c\`adl\`ag local $L^2$-martingale
with values in the Hilbert space $X$
(see Sections 14.5 and 14.6 in \cite{MePe}, Sections 2.2 and 2.3 in \cite{Rozo}, and \cite[Chapter 26]{Kal} for the scalar case).

We start with a simple but useful lemma which allows us to replace the operator $A$ by $A-\rho$ for any
$\rho\in \C$. In the applications of Theorem~\ref{thm:main} this is quite essential since one can usually
check the boundedness of the $H^\infty$-calculus only for $A-\rho$ with large $\rho\ge0$.

\begin{lemma}\label{lem:Atransl}
Assume $a\in L^1_{\rm loc}(\R_+)$ and $\rho\in \C$. Let $s\in L^1_{\rm loc}(\R_+)$ solve $s  -\rho a*s = a$.
Then  problem \eqref{eq:Volterra2} has a weak solution with
c\`adl\`ag/continuous paths almost surely if and only if
\eqref{eq:Volterra2} with $(a,A)$ replaced by $(s,A-\rho)$ has a weak solution with c\`adl\`ag/continuous paths.
\end{lemma}
It is well known that there is a unique function $s\in L^1_{\rm loc}(\R_+)$ with $s - \rho a*s = a$, see
Theorem~2.3.5 in \cite{GLS}.

\begin{proof}
Assume that \eqref{eq:Volterra2} with $(a,A)$ replaced by $(s,A-\rho)$ has a weak solution $v$ with
c\`adl\`ag/continuous paths. Set $u = v + \rho s*v$. The paths of $u$ inherit the c\`adl\`ag/continuous properties
of the paths of $v$ since $f*g\in C_b(\R)$ if $f\in L^1(\R)$ and $g\in L^\infty(\R)$. (The latter fact can be
proved approximating $f$ by continuous functions.) Moreover, the above identities yield $s*v = a*u$. Using also
that $v$ is a weak solution, we compute
\begin{align*}
 \lb u(t), x^*\rb &= \lb v(t), x^*\rb +  \rho  s*\lb v(t), x^*\rb
  =  \lb u_0, x^*\rb  + s* \lb v(t), A^* x^*\rb   + \lb L(t), x^*\rb\\
& = \lb u_0, x^*\rb  + a*\lb u(t), A^* x^*\rb   + \lb L(t), x^*\rb
\end{align*}
for all $x^*\in D(A^*)$. Hence, $u$ is a weak solution to \eqref{eq:Volterra2}.
The converse implication can be proved in a similar way.
\end{proof}

Our main result is the following sufficient condition for the existence and uniqueness of a solution
which has c\`adl\`ag/continuous paths. We write $\hat{a}$ for the Laplace transform of $a$.
\begin{theorem}\label{thm:main}
Assume the following conditions.
\begin{enumerate}[\rm (1)]
\item There is a number $\rho\in \R$ such that $A-\rho$ is a sectorial operator of angle $\phi_{A-\rho}<\pi/2$
and $-(A-\rho)$ has a bounded $H^\infty$-calculus.
\item $a\in L^1_{\rm loc}(\R_+)$ and $t\mapsto e^{-w_0 t}a(t)$ is integrable on $\R_+$ for some $w_0\in \R$.
\item\label{it:conditionimportant} There exist constants $\sigma,\phi, c>0$ and $w\in \R$ such that
  $\sigma + \phi_{A-\rho}<\pi/2$, $\phi>\phi_{A-\rho}$, $\hat{a}$ is holomorphic on
  $\{\lambda\in\C: \Re(\lambda)>w\}$, and for all $\lambda\in \C$ with  $\Re(\lambda)>w$
we have
\begin{enumerate}[\rm (i)]
\item $\lambda \hat{a}(\lambda)\in \Sigma_{\sigma}$ \ and \ $\hat{a}(\lambda)\in \Sigma_{\pi-\phi}$,
\item $|\lambda \hat{a}'(\lambda)|\leq c\, |\hat{a}(\lambda)|$.
\end{enumerate}
\end{enumerate}
Let $u_0:\Omega\to X$ be $\mathcal{F}_0$-measurable.
Then  \eqref{eq:volterraDet} possesses a resolvent $S$, the stochastic problem \eqref{eq:Volterra2}
has a unique weak solution $u$ given by
\[u(t) = S(t) u_0 + \int_0^t S(t-s) \, dL(s), \qquad t\ge0, \]
and $u$ has a modification with c\`adl\`ag/continuous trajectories almost surely
whenever the local $L^2$-martingale $L$ has c\`adl\`ag/continuous paths almost surely.
\end{theorem}

As announced the required existence of an $H^\infty$-calculus plays a crucial role in our approach.
The second sector condition in \eqref{it:conditionimportant}(i) and the assumption of \emph{1-regularity}
in \eqref{it:conditionimportant}(ii) are quite common in the theory of Volterra equations of parabolic type,
see Chapters~3 and 8 of \cite{Pruss93}. The first condition in \eqref{it:conditionimportant}(i) is needed
to derive the positive definiteness of the resolvent, as defined next.

The proof of Theorem~\ref{thm:main} relies on the following result which is a slight variation of Proposition~3 in
\cite{PZ14}. Before stating this result, we recall that a family of operators $(R(t))_{t\in \R}$ on $X$ is called
{\em positive definite} if $R(t) = R(-t)^*$  and
\[\sum_{m,n=1}^N \lb R(t_n-t_m)x_m, x_n\rb\geq 0.\]
for all $t,t_1, \ldots, t_N\in \R$, $x_1, \ldots, x_N\in X$ and $N\in\N$.

\begin{proposition}\label{prop:PZcond}
Assume that $(R(t))_{t\in \R}$ is a strongly continuous family of operators on $X$ such that $R(0) = I$ and
the family  $e^{-w |t|}R(t)$ is positive definite for some $w\in \R$.
If $L$ is a c\`adl\`ag (or continuous) local $L^2$-martingale with values in $X$, then the process
\[u(t) = R(t) u_0 + \int_0^t R(t-s) \, dL(s), \qquad t\ge0, \]
is c\`adl\`ag (or continuous) as well.
\end{proposition}
The proof of this fact uses that  the family $R$ has a dilation to a
strongly continuous group by Na{\u\i}mark's theorem (see Theorem~I.7.1 in \cite{Nagy10})
and an  argument from \cite{HauSei01,HausSei08}.

\begin{remark}
Proposition \ref{prop:PZcond} can be extended to larger classes of integrators $L:[0,\infty)\times \Omega\to X$.
The only properties needed in the proof are:
\begin{enumerate}
\item For every strongly continuous $f:\R_+\to \calL(X)$, the stochastic integral process $\int_0^\cdot f(s) \, dL(s)$ exists
and almost all paths are c\`adl\`ag (or continuous).
\item For any $B\in \calL(X)$, the following identity holds almost surely
\[B\int_0^t f(s) \, dL(s) = \int_0^t B f(s) \, dL(s), \ \ t\geq 0.\]
\end{enumerate}
\end{remark}

The proof of Theorem ~\ref{thm:main} will be divided into several steps. We first reduce the problem to the case
$\rho=0$. After that we will use the functional calculus to construct the resolvent and to show that it is positive
definite. The above proposition then implies the assertions.

\begin{proof}[Proof of Theorem \ref{thm:main}]
{\em Step 1: Reduction to $\rho=0$.}
By Lemma \ref{lem:Atransl}, it suffices to prove the result with $(a,A)$ replaced by $(s,A-\rho)$,
where $s\in L^1_{\rm loc}(\R_+)$ satisfies $s - \rho a*s = a$. Moreover, for  $\Re(\lambda)>w\ge w_0$, we have
\[|\hat{a}(\lambda)|\leq \int_0^\infty e^{-wt} |a(t)|\, dt\ \longrightarrow \ 0\]
as $w\to \infty$. Hence, for all sufficiently large $w$,  we can write
\[\hat{s}(\lambda) = \frac{\hat{a}(\lambda)}{1-\rho \hat{a}(\lambda)}, \ \ \  \ \Re(\lambda)>w.\]
 It is then  easy to check that also $s$ satisfies the assumptions of the theorem for a fixed
 (possibly larger) $w\ge0$, but one may have to  increase $\sigma$ and decrease $\phi$ a bit.
 In the following we can thus assume that $\rho=0$ and write $A$ instead of $A-\rho$.

\smallskip
{\em Step 2: Construction of the resolvent.}
Choose $\beta\in (\phi_A,\phi)$ such that $\beta+\sigma<\pi/2$. Let $\alpha = \beta \frac{2}{\pi}$.
It follows from Theorem~11.14 of \cite{KunWeis04} that $-A$ has a dilation to a multiplication operator $M$
on $L^2(\R;X)$ given by
\[Mf(\tau) =  -(i\tau)^{\alpha} f(\tau), \qquad \tau\in \R.\]
This means that there exists an isometric embedding $J:X\to L^2(\R;X)$ such that $J J^*$ is an orthogonal
projection from $L^2(\R;X)$ onto $J(X)$ and $J^* J = I$ on $X$ and for all
$\psi>\beta$ and $f\in H^{\infty}(\Sigma_\psi)$ we have
\begin{align}\label{eq:resolventdil}
f(-A) = J^* f(-M) J.
\end{align}

Set $a_{w}(t)=e^{-w t} a(t)$ with $w\geq 0$ from Step 1.
For each $\mu\in \C$, let $s_{w,\mu}$ be the unique solution to the equation
\begin{align}\label{eq:stau}
s_{w,\mu}(t) = e^{-wt} - \mu a_w*s_{w,\mu}(t).
\end{align}
The function $s_{w,\mu}$ is continuous. (See Theorems~2.3.1 and 2.3.5 of \cite{GLS}.)
We want to check that $\mu\mapsto s_{w,\mu}(t)$ belongs to $H^\infty(\Sigma_\psi)$ for $\psi\in(\beta,\phi)$.
We first  show the holomorphy of the map $\varphi_{w,t}:\mu\mapsto s_{w,\mu}(t)$  on $\C$ for fixed $t\geq 0$.

To this aim, take $\mu_0\in \C$ and $\varepsilon>0$. Set $B = \{\mu\in \C:|\mu-\mu_0|<\varepsilon\}$.
It enough to prove that $\mu\mapsto s_{\tilde{w},\mu}(t)$ is holomorphic on $B$ for a sufficiently large
$\tilde{w}\ge0$. Indeed, the uniqueness of \eqref{eq:stau} yields
$s_{w,\mu}(t) = e^{(\tilde{w}-w)t} s_{\tilde{w},\mu}(t)$band thus $\varphi_{w,t}$
 will also be holomorphic on $B$. Since $B$ is arbitrary, the holomorphy of $\varphi_{w,t}$ on $\C$ will then follow.
Take now $\tilde{w}$ such that $|\mu_0+\varepsilon|\, \| a_{\tilde{w}}\|_{L^1(\R_+)}<1$. By the proof of
Theorem~2.3.1 of \cite{GLS} the function $r_{\mu} = \sum_{j=1}^\infty (-1)^{j-1} (\mu a_{\tilde{w}})^{*j}$
converges in $L^1(\R_+)$ uniformly for $\mu\in B$ and $r_{\mu}$ solves
$r_{\mu}+\mu a_{\tilde{w}}*r_{\mu} = \mu a_{\tilde{w}}$. Hence, the map $B\ni \mu\mapsto r_{\mu}\in L^1(\R_+)$ is
holomorphic. Theorem~2.3.5 of \cite{GLS} also yields that
\[s_{\tilde{w},\mu}(t) = e^{-\tilde{w}t} - \int_0^t r_\mu(t-\tau) e^{-\tilde{w}\tau}\, d\tau, \qquad t\in \R_+.\]
The right-hand side is holomorphic in $\mu\in B$ for each  $t\in \R_+$
because integration with respect to the measure $e^{-\tilde{w}\tau}\, d\tau$ is a bounded linear functional
on $L^1(\R_+)$.

We next claim that there exists a constant $C>0$ such that $|s_{w,\mu}(t)|\leq C$ for all $t\geq 0$ and $\mu\in
\Sigma_{\psi}$, where $\psi\in(\beta,\phi)$.
Thanks to Corollary~0.1 and (the proof of) Proposition~0.1 of \cite{Pruss93}
it suffices to find a constant $K$ independent of $\mu$ such that
\[|\lambda \hat{s}_{w,\mu}(\lambda)| + |\lambda^2 \hat{s}'_{w,\mu}(\lambda)|\leq K, \qquad  \lambda\in \C_+.\]
(These results at first give the bound  on $s_{w,\mu}(t)$ only for a.e.\ $t$, but  $s_{w,\mu}$ is continuous.)
Since $\hat{s}_{w,\mu}(\lambda) = \frac{1}{\lambda+w} \frac{1}{1+\mu\hat{a}(\lambda+w)}$ by  \eqref{eq:stau}, we can compute
\begin{align*}
|\lambda \hat{s}_{w,\mu}(\lambda)|
&= \Big|\frac{\lambda}{\lambda+w}\Big| \Big|\frac{1}{1+\mu\hat{a}(\lambda+w)} \Big|
   \leq \sup\big\{|1+z|^{-1}: z\in \Sigma_{\pi - (\phi-\psi)}\big\}=:M_1,\\
|\lambda^2 \hat{s}'_{w,\mu}(\lambda)| &=
\Big| \frac{\lambda^2}{(\lambda+w)^2}\Big| \Big|\frac{(1+\mu\hat{a}(\lambda+w))
   + (\lambda+w) \mu \hat{a}'(\lambda+w)}{(1+\mu\hat{a}(\lambda+w))^2}\Big| \\
& \leq \Big|\frac{1}{1+\mu\hat{a}(\lambda+w)}\Big|
      + \frac{c\,|\mu \hat{a}(\lambda+w)|}{|1+\mu\hat{a}(\lambda+w)|^2}\\
& \leq M_1 + \sup\Big\{\Big|\frac{c z}{(1+z)^2}\Big|: z\in \Sigma_{\pi - (\phi-\psi)}\Big\}=:M_1+M_2
\end{align*}
for all $\lambda\in \C_+$. Here we employed the second part of condition \eqref{it:conditionimportant}(i)
several times and \eqref{it:conditionimportant}(ii) in the penultimate estimate. The claim follows.

We conclude that the map $\mu\mapsto s_{w,\mu}(t)$ belongs to $H^\infty(\Sigma_{\psi})$ for each $t\ge0$. Using the
$H^\infty$-calculus of $-A$, we  define $S_{w,A}(t) = s_{w,-A}(t)$ in $\calL(X)$ with norm less or equal $C_A C$.
To relate these operators to the desired resolvent, we further let $S_{w,M}(t)$ be the (multiplication) operator
on $L^2(\R;X)$ which is given by the map $\mu\mapsto s_{w,\mu}(t)$ and the functional calculus of $-M$. The norm
of $S_{w,M}(t)$ is bounded by $C$. Since the maps  $s_{w,\mu}$ are continuous,
$S_{w,M}(t)f$ is continuous in $L^2(\R;X)$ for $t\ge0$ if $f$ is a simple function.
By density and uniform boundedness, we infer that
 $t\mapsto S_{w,M}(t)$ is strongly continuous. Equation \eqref{eq:resolventdil} further yields
 \begin{equation}\label{eq:SMA}
S_{w,A}(t) = J^* S_{w,M}(t) J, \qquad t\geq 0.
\end{equation}

This identity and the strong continuity of $S_{w,M}$ imply that  $S_{w,A}$ is strongly continuous.
The operators  $S_{w,A}(t)$ and $A$ commute on $D(A)$ by the functional calculus, see e.g.\
Theorem~2.3.3 in \cite{Haase:2}. To derive the resolvent equation, we observe
\[
S_{w,M}(t)g = e^{-wt}g + a_w*(M S_{w,M} g)(t)
\]
for $g\in L^2(\R;X)$ and $t\ge0$ due to the definition of $S_{w,M}$ and  \eqref{eq:stau}.
This identity for $g=nR(n,M)Jx$ with $x\in D(A)$ and equation \eqref{eq:resolventdil} then imply
\begin{align*}
S_{w,A}(t)nR(n,A) x  & = J^* S_{w,M}(t)nR(n,M)Jx\\
& = J^*e^{-wt}nR(n,M)Jx + J^* a_w*(S_{w,M}nM R(n,M)Jx)(t)\\
& = e^{-wt}nR(n,A)x + a_w*(S_{w,A}n AR(n,A) x)(t).
\end{align*}
Letting $n\to \infty$ and using that $A$ and $S_{w,A}(t)$ commute on $D(A)$, we find
\begin{align}\label{eq:resolveq}
S_{w,A}(t)x  =  e^{-wt}x + a_w*(S_{w,A}Ax)(t)  = e^{-wt}x +  a_w*(AS_{w,A} x)(t) \, ds
\end{align}
for $t\ge0$. One now easily sees that
$(e^{wt}S_{w,A}(t))_{t\geq 0}$ is the resolvent of \eqref{eq:volterraDet}.

\smallskip
{\em Step 3: Positive definiteness.}
Let $\tilde{S}_{w,A}$ be the extension of $S_{w,A}$ given by $\tilde{S}_{w,A}(t) = S_{w,A}(-t)^*$.
Analogously we extend $\tilde{S}_{w,M}$ and $\tilde{s}_{w,\mu}$ to functions on $\R$.
Fix $t_1, \ldots, t_N\geq0$ and $x_1, \ldots, x_N\in X$. Setting $f_n = J x_n$, we infer from
\eqref{eq:SMA} that
\begin{align*}
\sum_{m,n=1}^N \lb \tilde{S}_{w,A}(t_n-t_m)x_m, x_n\rb
  = \int_{\R} \sum_{m,n=1}^N \lb \tilde{s}_{w,(i\tau)^{\alpha}}(t_n-t_m)f_m(\tau), f_n(\tau)\rb \, d\tau.
\end{align*}
Therefore, for the positive definiteness of $S_{w,A}(t)$ is suffices to prove that
the function $\tilde{s}_{w,(i\tau)^{\alpha}}$ is positive definite for a.e.\ $\tau\in \R$.

By the easy direction of Bochner's characterization it is enough to check that
$\F(\tilde{s}_{w,(i\tau)^{\alpha}})(\xi)\geq 0$ for all $\xi\geq 0$, where $\F$ denotes the Fourier transform.
The Fourier transform of $\tilde{s}_{w,(i\tau)^{\alpha}}$ satisfies
\begin{align*}
\F(\tilde{s}_{w,(i\tau)^{\alpha}})(\xi) & = \int_0^{\infty} \tilde{s}_{w,(i\tau)^{\alpha}}(t) e^{-it \xi} \, dt
 + \int_{-\infty}^{0} \overline{\tilde{s}_{w,(i\tau)^{\alpha}}(-t)} e^{-it \xi} \, dt\\
 & = 2\Re \int_0^{\infty} s_{w,(i\tau)^{\alpha}}(t) e^{-it \xi} \, dt
 = 2 \Re(\widehat{s}_{w,(i\tau)^{\alpha}}(i\xi)).
\end{align*}
If we extend $s$ and $a$ by zero to $t<0$, equation \eqref{eq:stau} yields
\begin{align*}
\Re \widehat{s}_{w,(i\tau)^{\alpha}}(i\xi)
&= \Re\Big( (1+(i\tau)^{\alpha} \hat{a}(w+i\xi))^{-1} (w+i\xi)^{-1} \Big)\\
&= \Re\Big( (1+ |\tau|^\alpha e^{\pm i\alpha\pi/2} \hat{a}(w+i\xi))^{-1} (w+i\xi)^{-1} \Big),
\end{align*}
where  $\pm$ is the sign of $\tau\in\R$.
Clearly, $\Re(z^{-1}) \geq 0$ if and only if $\Re(z)\geq 0$.
The number $z_0=\hat{a}(w+i\xi)) (w+i\xi)$ belongs to $\Sigma_\sigma$
by assumption \eqref{it:conditionimportant}. Since $e^{\pm i\alpha\pi/2}=e^{\pm i\beta}$,
 the condition $\sigma+\beta<\pi/2$ implies that $e^{\pm i\alpha\pi/2}z_0$ has a nonnegative
real part. Hence, $\Re\widehat{s}_{w,(i\tau)^{\alpha}}(i\xi)$ is nonnegative as required.

\smallskip
{\em Step 4: Conclusion.}
Since the resolvent  $S(t) := e^{wt} S_{w,A}(t)$ exists, the solution $u$ of \eqref{eq:Volterra2} is given by
\eqref{eq:mild}. Now as $\tilde{S}_{w,A}$ is positive definite, Proposition \ref{prop:PZcond} shows that $u$
has a version with the required properties.
\end{proof}

\begin{remark}
In Theorem~1 of \cite{PZ14} it is assumed that $A$ is self-adjoint and nonpositive.
In this paper the crucial property of the kernel $a$ is the inequality
$\Re(\lambda \hat{a}(\lambda))\ge0$
for all $\lambda\in \C_+$ with $\Re(\lambda)\geq w$ for some $w$, which corresponds to
$\sigma=\pi/2$ in our theorem. This sharp case is needed for the kernel $a(t)=t$
which leads to second order Cauchy problems such as the wave equation.
We cannot treat this case since we have to slightly enlarge sectors when working
with the $H^\infty$-calculus instead of the functional calculus for self-adjoint operators.

However, if one wants to use our Lemma~\ref{lem:Atransl} to extend Theorem~1 of \cite{PZ14}
to self-adjoint operators with   $\lb Ax, x\rb\leq \rho\|x\|^2$ for $x\in D(A)$ and some
$\rho>0$, then one has to impose the slightly stronger sector condition $\lambda \hat{a}(\lambda)\in
\Sigma_{\frac{\pi}{2} - \varepsilon}$ for some $\varepsilon>0$. In fact, by Step~1 of our proof
the shifting procedure requires this extra angle.
\end{remark}

\begin{remark}
Let $H$ be a separable Hilbert space and let $W_H$ be a  cylindrical Brownian motion on $H$.
An important special case is given by  $u_0 = 0$ and
\[L(t) = \int_0^t g \, dW_H,\] where $g\in L^2_{\rm loc}(\R_+;\calL_2(H,X))$ a.s.\ is measurable and adapted.
This process $L$ is a continuous local martingale. If the conditions of
Theorem~\ref{thm:main} hold, then the solution $u$ has a version with continuous paths. Moreover,
the solution formula \eqref{eq:mild} and  the Burkholder-Davis-Gundy estimate imply that
\[\big(\E(\sup_{t\in [0,T]} \|u(t)\|^p)\big)^{1/p} \leq C \|g\|_{L^p(\Omega;L^2(0,T;\calL_2(H,X)))}\]
for every $T<\infty$ and $p\in (0, \infty)$,
whenever the righthand side is finite. (See also \cite{HauSei01,HausSei08}.)
Here $C$ is a constant independent of $g$.
In \cite{HauSei01,HausSei08} is has been shown how one can use this result to obtain results
on exponential integrability of $\sup_{t\in [0,T]} \|u(t)\|^2$, and their methods extend to our setting.
\end{remark}

\section{Applications}

\subsection{Examples of kernels $a$\label{subsec:aex}}
In this section we present examples of kernels which satisfy the conditions of Theorem \ref{thm:main}.
Examples of sectorial operators $A$ with an $H^\infty$-calculus
 have been given in Section \ref{subsec:funccalc}.
We start with the arguably most prominent class of scalar kernels $a$.

\begin{example}\label{ex:standarda}
Let $a:(0,\infty)\to \R$ be given by $a(t) = t^{\beta-1}/\Gamma(\beta)$ with $\beta\in (0,2)$. Assume that the
operator $A$ is sectorial with
\[\phi_{A-\rho}<\min\{\tfrac{\pi}{2}(2-\beta), \tfrac{\pi}{2}\beta\}\]
for some $\rho$ and that $-(A-\rho)$ has an $H^\infty$-calculus.
Then the conditions of Theorem~\ref{thm:main} are fulfilled with $w=0$.

To check this claim, let  $\lambda\in \C_+$. Since  $\hat{a}(\lambda) = \lambda^{-\beta}$, we can compute
$|\lambda \hat{a}'(\lambda)| = \beta |\lambda|^{-\beta} = \beta|\hat{a}(\lambda)|$. Moreover,
$\lambda \hat{a}(\lambda)=\lambda^{1-\beta}$ belongs to $\Sigma_{\sigma}$ and
$\hat{a}(\lambda)=\lambda^{-\beta}$ to $\Sigma_{\pi-\phi}$
with $\sigma = |\beta-1|\tfrac{\pi}{2}$ and $\phi = \tfrac{\pi}{2}(2-\beta)$.
Our assumption on $\phi_{A-\rho} $ then yields  $\phi_{A-\rho} + \sigma<\pi/2$ and $\phi>\phi_{A-\rho}$.
\end{example}

We add a basic example from viscoelasticity discussed Section~5.2 of \cite{Pruss93}.

\begin{example}
Let $a(t) = \nu+\mu t$ with $\nu, \mu>0$ be the kernel arising in a Kelvin--Voigt solid.
Let $A$ be sectorial with   $\phi_A<\pi/2$ and let $-A$ possess an $H^\infty$-calculus.
We show  the conditions of Theorem \ref{thm:main}.

Let  $\lambda\in \C_+$ with $\Re \lambda >w$. We first observe that $\hat{a}(\lambda) = \frac{\nu}{\lambda} + \frac{\mu}{\lambda^2}$.
It suffices to check \eqref{it:conditionimportant}.  One has $|\lambda\hat{a}'(\lambda)|\leq 2\,|\hat{a}(\lambda)|$ for any choice $w\ge0$.
Take  $\sigma>0$ with $\phi_A+\sigma<\pi/2$ and set $\phi=\frac{\pi}{2}-\sigma >\phi_A$. Notice that
$\lambda\hat{a}(\lambda) =\nu + \frac{\mu}{\lambda}$ belongs to $\nu+ (B(0,\mu/w)\cap \C_+)$.
Hence,  $\lambda\hat{a}(\lambda)  \in \Sigma_{\sigma}$ for a fixed sufficiently large $w$.
This fact then implies that  $\hat{a}(\lambda)  \in \Sigma_{\sigma+\pi/2}=\Sigma_{\pi-\phi}$.
\end{example}

Our final example cannot be treated within our setting.
\begin{example}
Let $a(t) = t$. Then $\hat{a}(\lambda) = \frac{1}{\lambda^2}$ and so $\lambda \hat{a}(\lambda) = \lambda^{-1}$.
Hence we have to take $\sigma = \pi/2$, which contradicts the assumption $\phi_A + \sigma <\pi/2$.
\end{example}

\subsection{Illustration\label{subsec:illustrate}}
In this section we  present an example of a stochastic Volterra equation with all details.
This is an illustration how the results from the previous sections can be combined. One can easily
treat much larger classes of examples. We study the equation
\begin{equation}\label{eq:problem}
 u(t) = u_0 + \int_0^t a(t-s) A u(s) \, ds + L(t).
\end{equation}
with $a(t) = t^{\beta-1}/\Gamma(\beta)$ for any fixed $\beta\in (0,2)$ and
\[A u = \sum_{m,n=1}^d a_{m,n} u_{x_m, x_n} + \sum_{n=1}^d b_n u_{x_n} + cu.\]
We assume that
\begin{itemize}
\item $b_n,c\in L^\infty(\R^d)$,
\item $a_{m,n}\in C^{\varepsilon}_b(\R^d)$ for some $\varepsilon>0$,
\item $a_{mn} = a_{nm}$ are real valued and $\sum_{m,n=1}^N a_{m,n}(x)\xi_n \xi_m\geq \delta|\xi|^2$.
\end{itemize}
Let $D(A) = H^2(\R^d)$. The next result follows from Theorem \ref{thm:main}.

\begin{proposition}
Assume the above conditions and that $L$ is a local $L^2$-martingale with c\`adl\`ag/continuous paths almost surely.
Then \eqref{eq:problem} has a unique weak solution $u$
and $u$ has a modification with c\`adl\`ag/continuous trajectories almost surely.
\end{proposition}

\begin{proof}
Theorem~13.13 of \cite{KunWeis04} shows that $\lim_{\rho\to \infty} \phi_{A-\rho} = 0$ and that  $-(A-\rho)$ has a bounded $H^\infty$-calculus  for all
sufficiently large $\rho$. We choose $\rho\ge0$ so that $\phi_{A-\rho}<\min\{\tfrac{\pi}{2}(2-\beta), \tfrac{\pi}{2}\beta\}$.
Setting $\sigma = |\beta-1|\tfrac{\pi}{2}$ and $\phi = \tfrac{\pi}{2}(2-\beta)$,
the conditions of Theorem \ref{thm:main} hold due to Example~\ref{ex:standarda}.
\end{proof}


\end{document}